\documentclass[11pt,a4paper,reqno]{amsart}
\usepackage{clrscode}
\usepackage{mathrsfs,amssymb}

\usepackage{amsmath, amsthm}

\newtheorem{theorem}{Theorem}[section]

\newtheorem{corollary}[theorem]{Corollary}

\newtheorem{proposition}[theorem]{Proposition}

\newtheorem*{claim}{Claim}

\newcommand{\rinf}{\mathbb{R}^\infty}

\begin{document}

\title[Monoids acting by isometric embeddings]{A \v{S}varc-Milnor Lemma for \\ monoids acting by isometric embeddings}

\maketitle

\begin{center}

    ROBERT GRAY\footnote{
Centro de \'{A}lgebra da Universidade de Lisboa, Av.~Prof.~Gama Pinto, 2, 1649--003 Lisboa, Portugal.
Email \texttt{rdgray@fc.ul.pt}. This author's research is partially
supported by
FCT and FEDER, project POCTI-ISFL-1-143 of Centro de \'{A}lgebra da
Universidade de Lisboa, and by the project PTDC/MAT/69514/2006. This research
was partially conducted while this author was at the University
of St Andrews, where he was supported by an EPSRC Postdoctoral Fellowship.} \ and \ MARK KAMBITES\footnote{School of Mathematics, University of Manchester, Manchester M13 9PL, England. Email \texttt{Mark.Kambites@manchester.ac.uk}.
This author's research is supported by an RCUK Academic Fellowship.} \\
\end{center}

\begin{abstract}
We continue our programme of extending key techniques from geometric group
theory to semigroup theory, by studying monoids acting by isometric embeddings
on spaces equipped with asymmetric, partially-defined distance functions. The
canonical example of such an action is a cancellative monoid acting by
translation on its Cayley graph. Our main result is an extension of
the \v{S}varc-Milnor Lemma to this setting.
\end{abstract}

\section{Introduction}\label{sec_intro}

Over the past few decades, combinatorial group theory has been
increasingly influenced by its connections with geometry.
The resulting subject of \textit{geometric group theory} (see for
example \cite{Bowditch06,Bridson99,delaHarpe00})
is based on two principles. The first is that the use of
\textit{diagrams}
associated to a group can be a valuable aid to combinatorial reasoning with
generators and relations. The second is that a
group can be understood by studying way it acts, in a suitably controlled
way, upon a metric space. Arguably
the greatest power for understanding groups, though, comes from a synthesis
of these ideas, in which diagrams (and, by extension, groups themselves) are
endowed with a metric structure, and the groups made to act upon the resulting
spaces.

It is very natural to ask whether these geometric methods are particular to
groups, or if they have any potential for wider application. Finitely
generated semigroups and monoids are probably
the most obvious candidates for such a generalization, and indeed a
number of authors have employed geometric ideas in semigroup theory.
However, while some success has been enjoyed with ``diagrammatic'' techniques
(see for example \cite{Hoffmann06,Remmers80,Silva04}), the same cannot
yet be said for truly geometric methods, and
there is as yet no coherent subject of \textit{geometric semigroup theory}
studying
semigroups of isometries on metric spaces.
In fact, we contend that there are two fundamental obstructions to the
existence of such a subject.

The first concerns the way in which monoids act: a monoid acting
faithfully by bijections is necessarily group-embeddable, and monoids
in general are not group-embeddable. Hence, to study more general monoids
through their actions it is necessary to consider actions by more general
functions than just permutations. In the case of actions on spaces with
distance functions, this means that it is not sufficient to consider
actions by isometries. The second fundamental obstruction concerns the nature of the spaces acted
upon. In the group case, the essential philosophy (made precise in the
\textit{\v{S}varc-Milnor Lemma} \cite{Milnor68,Svarc55}) is that a group acting in
a suitably controlled way upon any metric space must resemble (more precisely,
be \textit{quasi-isometric} to) the space, and so properties of the group can be read off from properties of
the space and vice versa. In a monoid or semigroup, by contrast, distance is
neither symmetric (since there are no inverses) nor everywhere defined (since
there may be ideals). Hence, there is no hope that a general monoid will
``resemble'' a metric space. Instead, the development of a true
semigroup-theoretic analogue of geometric group theory will require the
study of spaces which have the flexibility to resemble monoids, that is,
spaces with asymmetric, partially-defined distance functions.

This paper forms part of an ongoing programme of
research, in which we seek to transfer techniques of geometric group theory to
semigroup theory, by replacing metrics with what we call
\textit{semimetrics}\footnote{Such
functions are ubiquitous in applied mathematics and arise with increasing
frequency also in pure mathematics, but terminology for them is not
standardized across the different areas in which they appear. Terms used
include \textit{premetric}, \textit{pseudo-metric}, \textit{quasi-metric}, \textit{quasi-semi-metric}
and \textit{extended quasi-metric}, but some of these are also used for other
generalizations of metrics.}.
In \cite{K_semimetric} we initiated this study, by introducing a natural
notion of quasi-isometry for these spaces, and proving an analogue
of the \v{S}varc-Milnor lemma for groups acting on them by isometries.
We also presented some applications in semigroup theory, arising
from the action of a \textit{Sch\"utzenberger group} on the corresponding
\textit{Sch\"utzenberger graph}. Here, we begin to extend this theory to cover
actions of \textit{monoids} on semimetric spaces. Specifically, we consider monoids
acting by \textit{isometric embeddings} on semimetric spaces. While still
not general enough to permit the study of monoids in absolute generality,
this setting does encompass many more monoids than just groups, including
most notably all \textit{cancellative} monoids. In future research we shall
explore the extent to which these methods can be extended yet further, to
consider yet more general maps and hence yet larger classes of monoids;
however, initial indications are that such results will become markedly more
technical once cancellativity conditions are dropped.

In addition to this introduction, this paper comprises three sections. In
Section~\ref{sec_semimetric} we briefly recall some definitions and
foundational results from \cite{K_semimetric}. We then consider actions
of monoids by isometric embeddings on semimetric spaces, identifying some key
``well-behavedness'' conditions. Section~\ref{sec_cayley} considers monoids
acting on their own Cayley graphs (viewed as semimetric spaces), and the
extent to which these actions satisfy the conditions of the previous section.
Finally, Section~\ref{sec_svarc} proves an analogue of the \v{S}varc-Milnor
lemma and gives some corollaries and example applications, including a
complete description of the quasi-isometry types of free products of free
monoids and finite groups.

\section{Semimetric Spaces and Monoid Actions}\label{sec_semimetric}

We denote by $\rinf$ the set
$\mathbb{R}^{\geq 0} \cup \lbrace \infty \rbrace$ of
non-negative real numbers with $\infty$ adjoined, equipped with the obvious
total order, addition and multiplication (leaving $0 \infty$ undefined).
A \emph{semimetric} on a set $X$ is a function
$d : X \times X \to \rinf$ satisfying:
\begin{itemize}
\item[(i)] $d(x,y)=0$ if and only if $x=y$; and
\item[(ii)] $d(x,z) \leq d(x,y) + d(y,z)$;
\end{itemize}
for all $x,y,z \in X$. A \textit{semimetric space} is a set equipped
with a semimetric. A point $x_0 \in X$ is called a
\textit{basepoint} for the space $X$ if
$d(x_0, y) \neq \infty$ for all $y \in X$. The space is called
\textit{strongly connected} if every point is a basepoint, that is,
if the distance function never takes the value $\infty$.

We define the distance between subsets of $X$ by
$$d(A,B) \ = \ \inf \lbrace d(a,b) \mid a \in A, b \in B \rbrace$$
for all $A, B \subseteq X$.

A \textit{path of length $n \in \mathbb{R}$ from $x$ to $y$} is a map $p : [0,n] \to X$
such that $p(0) = x$, $p(n) = y$ and $d(p(a), p(b)) \leq b-a$ for all
$0 \leq a \leq b \leq n$.
If $d(x, y) \neq \infty$ then a \textit{geodesic} from $x$ to $y$ is a path
of length $d(x,y)$ from $x$ to $y$. The semimetric space $X$ is called
\emph{geodesic} if for all $x,y \in X$ with $d(x,y) \neq \infty$ there
exists at least one geodesic from $x$ to $y$.

Let $x_0 \in X$ and let $r$ be a non-negative real number. The 
 \emph{out-ball} of radius $r$ based at $x_0$ is
\[
\overrightarrow{\mathcal{B}}_r(x_0) = \{ y \in X \mid d(x_0,y) \leq r \}.
\]
Dually, the \emph{in-ball} of radius $r$ based at $x_0$ is defined by
\[
\overleftarrow{\mathcal{B}}_r(x_0) = \{ y \in X \mid d(y,x_0) \leq r \},
\]
and the \emph{strong ball} of radius $r$ based at $x_0$ is
\[
\mathcal{B}_r(x_0) = \overrightarrow{\mathcal{B}}_r(x_0) \cap
 \overleftarrow{\mathcal{B}}_r(x_0).
\]
For $1 \leq \mu < \infty$, a subset $X'$ of $X$ is called $\mu$-\textit{quasi-dense} if every
point in $X$ is contained in the strong ball of radius $\mu$ around some
point in $X'$.

Let $f : X \to Y$ be a map between semimetric spaces. Then $f$ is an
\textit{isometric embedding} if $d(f(x),f(y)) = d(x,y)$ for all
$x, y \in X$; a surjective (and hence bijective) isometric embedding is
an \textit{isometry}. More generally, let 
$1 \leq \lambda < \infty$, $1 \leq \mu < \infty$ and $0 < \epsilon < \infty$ be constants.
The map $f$ is called a \emph{$(\lambda,\epsilon)$-quasi-isometric embedding},
and $X$ \textit{embeds quasi-isometrically in $Y$}, if
\[
\frac{1}{\lambda}\;d(x,y) - \epsilon \leq
d(f(x), f(y)) \leq \lambda d(x,y) + \epsilon
\]
for all $x,y \in X$.
If $f : X \to Y$ is a $(\lambda,\epsilon)$-quasi-isometric embedding
and its image is $\mu$-quasi-dense, then $f$ is 
called a \textit{$(\lambda,\epsilon,\mu)$-quasi-isometry}, and the spaces
$X$ and $Y$ are \textit{quasi-isometric}. Quasi-isometry
is an equivalence relation on the class of semimetric spaces
\cite[Proposition~1]{K_semimetric}.
A semimetric space is called \textit{quasi-metric} if it is quasi-isometric
to a metric space, or equivalently \cite[Proposition~2]{K_semimetric} if there are
constants $\lambda, \mu < \infty$ such that $d(x,y) \leq \lambda d(y,x) + \mu$
for all points $x$ and $y$.

Now let $M$ be a monoid acting by isometric embeddings on a semimetric
space $X$. We say that the action is \textit{cobounded} if there is a strong
ball $B$ of finite radius such that $(mB)_{m \in M}$ covers $X$.
We say that the action is \textit{outward proper} if for every out-ball $B$
of finite radius the set  
$\lbrace m \in M \mid d(B, mB) = 0 \rbrace$
is finite. Note that our definitions of outward proper and cobounded
coincide with the usual notions of proper and cobounded in the special
case of a group action on a metric space. Also, notice that
coboundedness of the action is sufficient to ensure (if we do not assume
this \textit{a priori}) that $M$ is acting as a monoid, that is, that the
identity element of $M$ acts as the identity function on $X$.

In addition to the preceding conditions, which are relatively straightforward
generalisations of conditions imposed in the group case, we will also need to
impose a fundamentally semigroup-theoretic restriction to ensure that the
right ideal structure of the monoid is reflected in the space. If $x_0 \in X$
is a basepoint, we say that the action is \textit{idealistic at $x_0$} if
\[
d(mx_0,nx_0) < \infty
\Rightarrow
nM \subseteq mM
\]
for all $m,n \in M$.
We say that the action is \textit{idealistic} if it is idealistic at some
basepoint. Notice that in the special case that $M$ is a group, $mM \subseteq nM$ for
all $m, n \in M$, so the action is idealistic exactly if the space has
a basepoint.

\section{Cayley Graphs of Monoids and Cancellative Monoids}\label{sec_cayley}

Let $M$ be a monoid generated by a finite set $S$. Then we may define
a semimetric on $M$ by setting $d_S(x,y)$ to be the shortest length
of a word $w$ over the generating set $S$ such that $xw = y$ in $M$, or
$\infty$ if there is no such word. With this notion of distance, the
natural action of $M$ on itself by left translation is an action by
\textit{contractions} (maps which do not increase distance). The action
is by isometric embeddings exactly if $M$ is left cancellative

The semimetric space $M$ is clearly not (unless $M$ is trivial) geodesic.
We extend $M$ to a
geodesic semimetric space $\Gamma_S(M)$ by ``stitching in'', for each
element $m \in M$ and generator $s \in S$, a copy of the open interval
$(0,1)$ between $m$ and $ms$, which we view as an \textit{edge} from $m$
to $mx$.
The semimetric is defined so that point $\mu$ in this interval is at distance
$\mu$ from $m$ in both directions, but at distance $1- \mu$ from $mx$ in
only one direction.

Formally, our new semimetric space, which we denote $\Gamma_S(M)$ and
call the \textit{(continuous) Cayley graph of $M$ with respect to $S$},
has point
set $M \cup (M \times S \times (0,1))$ and distance function defined by
\begin{itemize}
\item $d(m,n) = d_S(m,n)$ for all $m,n \in M$;
\item $d((m,(n,y,\nu)) = d(m,n) + \nu$ for all $m,n \in M$, $y \in S$, $\nu \in (0,1)$;
\item $d((m,x,\mu), n) = \min(\mu + d(m,n), (1-\mu) + d(mx, n))$ for all $m,n \in M$, $x \in S$, $\mu \in (0,1)$;
\item $d((m,x,\mu), (m,x,\nu)) = |\mu - \nu|$ for all $m \in M$, $x \in S$, $\mu, \nu \in (0,1)$;
\item $d((m,x,\mu), (n,y,\nu)) = d((m,x,\mu), n) + \nu$ for all $m, n \in M$, $x,y \in S$, $\mu, \nu \in (0,1)$ such
that $m \neq n$ or $x \neq y$.
\end{itemize}
Note that this semimetric is slightly different from that used for Cayley
graphs in \cite{K_semimetric}; the distinction is immaterial in the group case,
but very important when considering monoids, where the definition used here
is necessary to allow the following elementary result.

\begin{proposition}\label{prop_monoidcayleyinclusion}
The inclusion map from $M$ to $\Gamma_S(M)$ is an isometric embedding
and a quasi-isometry.
\end{proposition}
\begin{proof}
That the map preserves distances, and hence is an isometric embedding
and a quasi-isometric embedding, is immediate from the definition of the
semimetric in $\Gamma_S(M)$. Moreover, it is also immediate that any
point $(m,x,\mu)$ on an edge lies in the strong ball of radius $1$ around
$m$, so that $M \subseteq \Gamma_S(M)$ is
quasi-dense and the map is a quasi-isometry.
\end{proof}

Note that while the semimetric spaces $M$ and $\Gamma_S(M)$ depend upon
the choice of finite generating set $S$, their quasi-isometry type does not
\cite[Proposition~4]{K_semimetric}, and hence is an isomorphism invariant of
$M$. Thus, provided $M$ admits a finite generating set, it makes sense to
speak of a semimetric space being quasi-isometric to the abstract monoid $M$.

For any monoid $M$, we can extend the left translation action of $M$ on
itself to an action on $\Gamma_S(M)$, by defining $p(m,x,\mu) = (pm,x,\mu)$
for all $p,m \in M$, $x \in S$ and $\mu \in (0,1)$. Recall (from for
example \cite{Silva04}) that a monoid
$M$ has \textit{finite (right) geometric type} (also called \textit{bounded
indegree}) if for every $b, c \in M$ there are only finitely many elements $a$
satisfying $ab = c$. Note in particular that this condition is satisfied
by right cancellative monoids.

\begin{proposition}\label{prop_actiononcayley}
Let $M$ be a left cancellative monoid generated by a finite set $S$. Then
$M$ acts on $\Gamma_S(M)$ by isometric embeddings, and the action is
cobounded and idealistic at the identity of $M$. If $M$ has finite geometric
type then the action is outward proper.
\end{proposition}
\begin{proof}
That the action is by isometric embeddings follows easily from the
definitions and the fact that $M$ is left cancellative.

Let $e$ denote the identity of $M$. Then the strong ball of radius $1$
around $e$ contains $e$ and all points of the form $(e,x,\mu)$ with
$x \in S$ and $\mu \in (0,1)$. Clearly every point in $\Gamma_S(M)$ is a
translate of such a point by an element of $M$, so the action is cobounded.

If $m, n \in M$ are such that $d(me,ne) = d_S(m,n) < \infty$ then there is
a word $w$ over the generating set $S$ such that $n = mw$. But now 
$n \in mM$ so that $nM \subseteq mM$. Thus, the action is idealistic at $e$.

Now suppose $M$ has finite geometric type. Since $e$ is a
basepoint, every out-ball of finite radius is contained in an out-ball of
finite radius around $e$, so it will suffice to assume $C$ is an out-ball
of finite radius around $e$, and show that the set
$\lbrace m \in M \mid d(C,mC = 0) \rbrace$
is finite. Let $B$ be any out-ball around $e$ of strictly larger radius than
$C$. Then if $m$ is such that $d(C,mC) = 0$ it is easily seen that
$B \cap mB \neq \emptyset$, so
it will suffice to show that
$$Q = \lbrace m \in M \mid B \cap mB \neq \emptyset \rbrace$$
is finite. Suppose $m \in Q$. Then we may choose $b_m,c_m \in B$ such that
$mb_m = c_m$. If either $b_m$ or $c_m$ is not in
$M$ then it follows from the definition of the action that we have
$b_m = (b', x, \mu)$ and $c_m = (c', x, \mu)$ for some $x \in S$, $\mu \in (0,1)$
and $b', c' \in M$ with $mb' = c'$. Moreover, since all paths from $e$ to $b_m$
and $c_m$ must lead through $b'$ and $c'$, we have $b', c' \in B$. Thus, replacing
$b_m$ and $c_m$ with $b'$ and $c'$ if necessary, we may assume without loss of
generality that $b_m, c_m \in B \cap M$. Now by the finite geometric type
assumption, there are only finitely many elements $x \in M$ satisfying $xb_m = c_m$.
Hence, the map
$$Q \to (B \cap M) \times (B \cap M), \ m \mapsto (b_m, c_m)$$
is finite-to-one. But since $B$ has finite radius and $M$ is finitely
generated, it is easily seen that $B \cap M$ is finite, and so
$Q$ is finite.
\end{proof}

One case of the above proposition deserves particular note.
\begin{corollary}\label{cor_cancaction}
A cancellative monoid acts outward properly, coboundedly and idealistically
by isometric embeddings on its Cayley graph.
\end{corollary}

\section{\v{S}varc-Milnor Lemma for Isometric Embeddings}\label{sec_svarc}

The aim of this section is to establish the following theorem, which is an
extension of the \v{S}varc-Milnor Lemma \cite{Milnor68,Svarc55} to the
setting of monoids acting by isometric embeddings on semimetric spaces. The
proof is adapted from our proof of a corresponding result for groups acting
by isometries on semimetric spaces \cite{K_semimetric}, which in turn is
based on the standard proof for groups acting on metric spaces (see
for example \cite{delaHarpe00}). However, extra steps are
required to handle the extra complications arising from the right ideal
structure of the monoid and the directedness of the space acted upon.

\begin{theorem}[\v{S}varc-Milnor Lemma for Isometric Embeddings]
\label{thm_svarc}
Let $M$ be a monoid acting idealistically, outward properly and coboundedly
by isometric embeddings on a geodesic semimetric space $X$. Then $M$ is
finitely generated and quasi-isometric to $X$.
\end{theorem}
\begin{proof}
Let $x_0 \in X$ be a basepoint such that the action is idealistic at
$x_0$. Since the action is cobounded, we may choose a strong
ball $D$ of finite radius such that $(mD)_{m \in M}$ covers $X$. In
particular, there exists $g \in M$ and $x_0' \in D$ with $x_0 = g x_0'$. Let
$p \in X$ be arbitrary. Since the action is by isometric embeddings and
$x_0$ is a basepoint, we have 
\[
d(x_0',p) = d(gx_0', gp) = d(x_0, gp) < \infty.
\] We conclude that $x_0'$ is a basepoint. Since
$x_0$ and $x_0'$ are both basepoints, and $D$ is a strong ball of finite
radius containing $x_0'$ whose translates cover $X$, it follows easily that
there is a strong ball $B \supseteq D$ of finite radius based at $x_0$
whose translates cover $X$. Let $R$ be the radius of $B$.

Now let
\[
S = \{ m \in M \mid d(B,mB) = 0 \}.
\]
Since the strong-ball $B$ is contained in an out-ball of the same radius, and
$M$ is acting outward property, the set $S$ is finite. Clearly $e \in S$,
where $e$ denotes the identity element of $M$.  

Let $C = \overrightarrow{\mathcal{B}}_{5R}(x_0)$, noting that $B \subseteq C$
and define
\[
Q = \{ mB \mid d(B, mB) \neq 0 \ \mbox{and} \ d(C,mB)=0 \}.
\]
Note that $Q$ is finite, since it is contained in $\{ mB \mid d(C,mC)=0 \}$,
which is finite since the action is outward proper. Hence, we may choose
a positive real number $r$ such that $r < R$ and $r < d(B,mB)$ for
every $mB \in Q$.

\begin{claim}
For all $h \in M$ if $d(B,hB)<r$ then $d(B,hB)=0$.
\end{claim}
\begin{proof}[Proof of Claim]
Suppose, seeking a contradiction, that $d(B,hB)<r$ but $d(B,hB) \neq 0$. Since   $d(B,hB)<r$ there exist $u \in B$ and $v \in hB$ with $d(u,v)<r$. Since $u \in B$ we have $d(x_0,u) \leq R$. Therefore
\[
d(x_0,v) \leq d(x_0,u) + d(u,v) < R + r \leq 2R < 5R
\]
so that $v \in C = \overrightarrow{\mathcal{B}}_{5R}(x_0)$.
Thus $v \in C \cap hB$ so $d(C,hB)=0$ and $hB \in Q$. But by the choice
of $r$ it now follows that $r < d(B,hB)$, giving the required contradiction
and completing the proof of the claim.
\end{proof}

\begin{claim}
For all $m,n \in M$, if $d(mB,nB) < r$ then $n = mu$ for some $u \in S$.
\end{claim}
\begin{proof}[Proof of Claim]
Since $d(mB,nB) < r < \infty$, we may choose $a, b \in B$ such that
$d(ma, nb) < r$. Now since $B$ is a strong ball of radius $R$, and the
action is distance-preserving, we have
\begin{align*}
d(m x_0, n x_0) &\leq d(m x_0, ma) + d(ma, nb) + d(nb, n x_0) \\
&= d(x_0, a) + d(ma,nb) + d(b, x_0) \\
&\leq R + r + R \\
&< \infty.
\end{align*}
Since the action is idealistic at $x_0$, this means that there exists
$u \in M$ with $mu = n$.

Now using again the fact that the action is distance preserving, we have
\begin{eqnarray*}
d(B,uB) 
 & = &
\inf \{  d(y,uz) : y, z \in B \} \\
 & = &
\inf  \{ d(my, muz) : y, z \in B \} \\
 & = &
d(mB, muB) \\
 & = &
d(mB, nB) \\
 & < &
r. 
\end{eqnarray*}
Hence, by the previous claim, we have $d(B,uB) = 0$, which by definition
gives $u \in S$, completing the proof of the claim. 
\end{proof}

Now choose a positive real number $l < r$. We claim that $S$ generates $M$ and that for all $m \in M$
\[
d_S(e,m) \leq \frac{1}{l}d(x_0,m x_0) +1
\]
where $e$ is the identity element of the monoid $M$.
To see this, let $m \in M$ be arbitrary. Since $x_0$ is a basepoint,
$d(x_0, mx_0)$ is finite. Since the semimetric space $X$ is
geodesic, there is geodesic from $x_0$ to $m x_0$, that is, a map
$p : [0, d(x_0, m x_0)] \to X$ such that $p(0) = x_0$, $p(d(x_0, mx_0)) = m x_0$
and $d(p(a), p(b)) \leq b - a$ for all
$0 \leq a \leq b \leq d(x_0, mx_0)$. Let $k$ be the integer part of
$\frac{1}{l} d(x_0, mx_0)$,
and for $1 \leq i \leq k$ define $x_i = p(il)$. Then for $0 \leq i < k$.
we have
$$d(x_i,x_{i+1}) = d(p(il), p((i+1)l)) \leq (i+1)l - il = l.$$
If we set $x_{k+1} = m x_0$ then recalling the definition of $k$ we also have
$$d(x_k,x_{k+1}) = d(p(kl), p(d(x_0, mx_0)) \leq d(x_0, mx_0) - kl < l.$$
Since the translates of
$B$ cover the space $X$, we may choose $m_0, \dots, m_{k+1} \in M$ such that
each $x_i \in m_i B$. Clearly, we may assume $m_0 = e$ and $m_{k+1} = m$.

Now for $0 \leq i \leq k$ 
\[
d(m_iB, m_{i+1}B) \leq d(x_i, x_{i+1}) \leq l < r,
\]
and so by the above claim there exists $s_i \in S$ with $m_{i+1} = m_i s_i$. But then
\[
m 
= m_{k+1}
= m_k s_k
= m_{k-1} s_{k-1} s_k
= \ldots = m_0 s_0 s_1 \ldots s_k
= s_0 s_1 \ldots s_k \in \langle S \rangle
\]
since $m_0 = e$.

We have written an arbitrary element $m \in M$ as a product of elements of $S$,
which proves the claim that $S$ generates the monoid $M$. 
Moreover, we have written $m$ as a product of $k+1$
generators from $S$, and $k$ was defined to be the integer part of
$\frac{1}{l} d(x_0, m x_0)$, so
\begin{align}
\label{eqn_bound}
d_S(e,m) \leq k+1 \leq \frac{1}{l} d(x_0,m x_0) + 1
\end{align}

Now let $m_1, m_2 \in M$ be arbitrary. If $d_S(m_1,m_2) = \infty$ then
by definition $m_2 M$ is not contained in $m_1 M$. Since the action
is idealistic at $x_0$, this means that
$d(m_1 x_0, m_2 x_0) = \infty$. In particular, we have
\[
d_S(m_1,m_2) \leq \frac{1}{l} d(m_1 x_0, m_2 x_0) + 1.
\]
On the other hand, if $d_S(m_1,m_2) < \infty$ then we can write $m_2 = m_1 n$ for some $n \in M$. Since
$M$ acts on $X$ by isometric embeddings and on itself by contractions,
applying equation \eqref{eqn_bound} we obtain
\begin{eqnarray*}
d_S(m_1,m_2) 	 & = &     	 d_S(m_1, m_1n) \\ 
				 & \leq &	 d_S(e,n) \\
				& \leq 	& \frac{1}{l}d(x_0,n x_0) +1 \\
				& = 		& \frac{1}{l}d(m_1 x_0, m_1 n x_0) +1 \\ 	
				& = 		& \frac{1}{l}d(m_1 x_0, m_2 x_0) +1. 
\end{eqnarray*}

Now let $\lambda = \max \{ d(x_0, sx_0) \mid s \in S \}$; since $x_0$ is a
basepoint and $S$ is finite, this maximum exists and is finite.
We claim that for all $m,n \in M$ we have 
\[
d(mx_0, nx_0) \leq \lambda d_S(m,n).
\]
Indeed, when $d_S(m,n) = \infty$ this is obviously true. Otherwise we can
write $n = m s_1 \ldots s_k$ where $k = d_S(m,n)$ and $s_1, \dots, s_k \in S$.
Now applying the triangle inequality and the fact that the action is by isometric embeddings we obtain
\begin{eqnarray*}
				& 		& d(mx_0, nx_0) 	\\
				& = 		& d(mx_0, ms_1 \ldots s_k x_0) \\
				& = 		& d(x_0, s_1 \ldots s_k x_0) \\
				& \leq 	& d(x_0, s_1 x_0) + d(s_1 x_0, s_1 s_2 x_0) + \ldots + d(s_1 s_2 \ldots s_{k-1} x_0, s_1 s_2 \ldots s_{k-1} s_k x_0) \\
				& = 		& d(x_0, s_1 x_0) + d(x_0, s_2 x_0) + \ldots + d(x_0, s_k x_0) \\
				& \leq 	& \lambda k = \lambda d_S(m,n). 
\end{eqnarray*}

Now consider the mapping $f: M \rightarrow X$ defined by $m \mapsto m x_0$. It follows from the observations above that
\[
d_S(m_1,m_2) \leq \frac{1}{l} d(f(m_1), f(m_2)) +1
\]
and also
\[
d(f(m_1),f(m_2)) \leq \lambda d_S(m_1,m_2)
\]
for all $m_1, m_2 \in M$. Moreover, given $x \in X$, since $(\alpha B)_{\alpha \in M}$ covers $X$ we conclude that there exists $h \in M$ with $x \in hB$, and thus
\[
\max(d(f(h),x),d(x,f(h))) \leq R.
\]
Hence $M$ and $X$ are quasi-isometric. 
\end{proof}

Combining Theorem~\ref{thm_svarc} with Proposition~\ref{prop_actiononcayley}
yields the following characterisation of the property of finite generation
for left cancellative monoids of finite geometric type (and so in particular
for cancellative monoids).

\begin{corollary}\label{cor_fingenchar}
A left cancellative monoid of finite geometric type is finitely generated if and only if it acts
outward properly, coboundedly, and idealistically by isometric embeddings on
a geodesic semimetric space.
\end{corollary}

\begin{corollary}
Let $M$ be a finitely generated monoid. Then $M$ is a group if and only if
$M$ acts outward properly, coboundedly and idealistically by isometric
embeddings on a quasi-metric space. If $M$ is a group then every semimetric
space on which it acts outward properly, coboundedly and idealistically by 
isometric embeddings is quasi-metric.
\end{corollary}
\begin{proof} 
Suppose $M$ acts properly, coboundedly and
idealistically by isometric embeddings on a geodesic quasi-metric space $X$.
Then by 
Theorem~\ref{thm_svarc}, $M$ is quasi-isometric to $X$. By
\cite[Corollary~1]{K_semimetric},
quasi-metricity is a quasi-isometry invariant, so it follows that $M$
itself is quasi-metric. By \cite[Proposition~8]{K_semimetric} this means
that $M$ is right simple, but a right simple monoid must be a group.

Conversely, if $M$ is a group then it is certainly cancellative, so by
Corollary~\ref{cor_cancaction} it admits an outward proper, cobounded,
idealistic action by isometric embeddings (which must in fact be isometries)
on a geodesic semimetric space. By Theorem~\ref{thm_svarc}, any space on
which $M$ so acts must be quasi-isometric to $M$, and hence quasi-metric.
\end{proof}

Recall that a subsemigroup $S$ of a semigroup $T$ is called \textit{left
unitary} if whenever $s \in S$ and $t \in T$ are such that $st \in S$, we
have also $t \in S$. The following theorem, which provides a generalisation of the well-known
fact that a finitely generated group is quasi-isometric to each of its
finite index subgroups, is an typical illustration of how Theorem~\ref{thm_svarc}
may be applied.

\begin{theorem}\label{thm_submonoid}
Let $M$ be a left unitary submonoid of a finitely generated
left cancellative monoid $N$ of finite geometric type, and suppose
there is a finite set $P \subseteq N$ of right units such that $MP = N$. Then $M$ is finitely generated and
quasi-isometric to $N$.
\end{theorem}
\begin{proof}
Consider the left translation actions of both $N$ and $M$ on $\Gamma_S(N)$.
By Proposition~\ref{prop_actiononcayley} the action of $N$ is outward proper and by
isometric embeddings, from which it follows immediately that the action of
$M$ is outward proper and by isometric embeddings.

Since
the action of $N$ is cobounded, there is a strong ball $B$ of finite radius
whose translates by elements of $N$ cover $\Gamma_S(N)$. By a simple argument
(or from the proof of Proposition~\ref{prop_actiononcayley}) we may assume
that this ball is centred at the identity $e$ of $N$. Since $P$ is finite
and consists of right units, it is contained in a strong ball of finite radius
around $e$ in $N$, and hence (since the embedding of $N$ into $\Gamma_S(N)$ is a
quasi-isometry) also in $\Gamma_S(N)$. Consider now the
set $PB$. Then for any $x \in P$ and $b \in B$ we have
$$d(e,xb) \leq d(e,x) + d(x,xb) = d(e,x) + d(e,b)$$
and similarly
$$d(xb,e) \leq d(xb,x) + d(x,e) = d(b,e) + d(x,e).$$
Since $P$ and $B$ are both contained in strong balls of finite radius
around $e$, it follows that $PB$ is
contained in a strong ball ($C$ say) of finite radius around $e$. Now for
any point $y \in \Gamma_S(N)$ we have $y = nb$ for some $n \in N$ and
$b \in B$. But $n = mx$ for some $m \in M$ and $x \in P$. Hence,
$y = mxb \in mPB \subseteq mC$, so the translates of $C$ by elements of
$M$ cover $\Gamma_S(N)$, and the action of $M$ is cobounded.

Finally, we claim that the action of $M$ is idealistic at $e$, considered
as a point in $\Gamma_S(N)$. Indeed, suppose
$m, n \in M$ are such that $d(me, ne) < \infty$ in $\Gamma_S(N)$.
By Proposition~\ref{prop_actiononcayley}, the action of $N$ is idealistic
at $e$,
so we have
$nN \subseteq mN$. Then there is an $s \in N$ such that $ms = n$. Since
$m \in M$, $ms = n \in M$ and $M$ is left unitary, we deduce that 
$s \in M$, whereupon $nM \subseteq mM$, as required.

Thus Theorem~\ref{thm_svarc} applies, and tells us that $M$ is finitely
generated and quasi-isometric to $\Gamma_S(N)$, which by 
Proposition~\ref{prop_monoidcayleyinclusion} is quasi-isometric to $N$.
\end{proof}

Theorem~\ref{thm_submonoid} allows us, for example, to completely describe
the quasi-isometry types of free products of finitely generated free
monoids and finite groups.

\begin{corollary}
Let $F$ be a finitely generated free monoid of rank $r$ and $G$ a finite
group. Then the free product $F * G$ is quasi-isometric to a free monoid of rank $r|G|$.
\end{corollary}
\begin{proof}
It is well known (see for example \cite[Lemma~1]{Cheng82}) that the free
product $F * G$ is cancellative.
Let $\lbrace f_1, \dots, f_r \rbrace$ be the free generating set for $F$.
It is easy to show that every element of $F*G$ can be written uniquely as
an alternating product of the form
$$g_0 x_1 g_1 x_2 g_2 \dots x_n g_n$$
where $n \geq 0$, each $g_i \in G$ and each $x_i \in \lbrace f_1, \dots, f_r \rbrace$.
Let $M \subseteq F*G$ be the set of elements which admit decompositions as
above with $g_n$ the identity element of $G$. It is readily seen that $M$ is
a left unitary
submonoid of $F*G$, and it also follows easily from the uniqueness of the
above decompositions that $M$ is freely generated by the elements of the
form $g f_i$ for $g \in G$ and $1 \leq i \leq r$, and hence is free of rank
$r|G|$. Finally, $G$ is a finite set of right units in $F*G$, and $MG = F*G$.
Hence, by Theorem~\ref{thm_submonoid} we may conclude that $F*G$ is
quasi-isometric to $M$, which is a free monoid of rank $r|G|$.
\end{proof}

\bibliographystyle{plain}

\begin{thebibliography}{10}

\bibitem{Bowditch06}
B.~H. Bowditch.
\newblock {\em A course on geometric group theory}, volume~16 of {\em MSJ
  Memoirs}.
\newblock Mathematical Society of Japan, Tokyo, 2006.

\bibitem{Bridson99}
M.~R. Bridson and A.~Haefliger.
\newblock {\em Metric spaces of non-positive curvature}, volume 319 of {\em
  Grundlehren der Mathematischen Wissenschaften [Fundamental Principles of
  Mathematical Sciences]}.
\newblock Springer-Verlag, Berlin, 1999.

\bibitem{Cheng82}
C.~C. Cheng and R.~W. Wong.
\newblock Hereditary monoid rings.
\newblock {\em Amer. J. Math.}, 104(5):935--942, 1982.

\bibitem{delaHarpe00}
P.~de~la Harpe.
\newblock {\em Topics in geometric group theory}.
\newblock Chicago Lectures in Mathematics. University of Chicago Press,
  Chicago, IL, 2000.

\bibitem{K_semimetric}
R.~Gray and M.~Kambites.
\newblock Groups acting on semimetric spaces and quasi-isometries of monoids.
\newblock {\tt arXiv:math.GR/0906.0473}, 2009.

\bibitem{Hoffmann06}
M.~Hoffmann and R.~M. Thomas.
\newblock A geometric characterization of automatic semigroups.
\newblock {\em Theoret. Comput. Sci.}, 369(1-3):300--313, 2006.

\bibitem{Milnor68}
J.~Milnor.
\newblock A note on curvature and fundamental group.
\newblock {\em J. Differential Geometry}, 2:1--7, 1968.

\bibitem{Remmers80}
J.~H. Remmers.
\newblock On the geometry of semigroup presentations.
\newblock {\em Adv. in Math.}, 36(3):283--296, 1980.

\bibitem{Silva04}
P.~V. Silva and B.~Steinberg.
\newblock A geometric characterization of automatic monoids.
\newblock {\em Q. J. Math.}, 55(3):333--356, 2004.

\bibitem{Svarc55}
A.~S. {\v{S}}varc.
\newblock A volume invariant of coverings.
\newblock {\em Dokl. Akad. Nauk SSSR (N.S.)}, 105:32--34, 1955.

\end{thebibliography}

\def\cprime{$'$} \def\cprime{$'$}

\end{document}